\documentclass[10pt, oneside]{amsart}

\usepackage{geometry}
\usepackage{parskip}
\usepackage{amsfonts, amsthm, amsmath, amssymb, amsfonts}
\usepackage{mathrsfs}
\usepackage{graphicx}
\usepackage{xtab}
\usepackage{MnSymbol}
\usepackage{mathtools} 
\graphicspath{ {./images/} }
\usepackage{wrapfig}
\usepackage{subcaption}
\usepackage{url}

\usepackage{fancyhdr} 
\usepackage{color} 
\usepackage{float}
\usepackage[hidelinks]{hyperref}
\usepackage[capitalise]{cleveref}
\usepackage{siunitx}
\usepackage{booktabs}
\usepackage[font={small}]{caption}
\usepackage{esint}
\usepackage{tikz}
\usepackage{pgfplots}
\usepgfplotslibrary{colormaps}
\pgfplotsset{compat=1.18}
\usepackage{listings}
\usepackage{mathdots}
\usepackage{nicefrac}

\theoremstyle{plain} 
\newtheorem{thm}{Theorem}[section]
\newtheorem{bigthm}{Theorem}[section]

\newtheorem{lem}[thm]{Lemma}
 
\newtheorem{prop}[thm]{Proposition}

\newtheorem*{thm*}{Theorem} 
\newtheorem*{prop*}{Proposition} 
\newtheorem{notation}[thm]{Notation}

\theoremstyle{definition} 
\newtheorem{defn}[thm]{Definition}

\theoremstyle{remark} 
\newtheorem{rem}{Remark}

\usepackage{setspace}

\newcommand{\mbb}[1][]{\mathbb{#1}}
\newcommand{\bN}{\mbb[N]}

\newcommand{\bR}{\mbb[R]}

\newcommand{\bZ}{\mbb[Z]}

\newcommand{\inv}{^{-1}}
\newcommand{\nton}{\bN\rightarrow\bN}
\newcommand{\ol}[1][]{\overline{#1}}
\newcommand{\pres}[2][]{\ensuremath{\left\langle #1 \middle| #2 \right\rangle}}
\newcommand{\hpres}[2][]{\ensuremath{\left\langle #1 \middle\| #2 \right\rangle}}
\newcommand{\ra}{\rightarrow}

\newcommand{\brac}[1][]{\ensuremath{\left( #1\right)}}

\newcommand{\setst}[2][]{\ensuremath{\left\{{#1}\,\middle|\,{#2}\right\}}}
\newcommand{\del}{\partial}
\DeclareMathOperator{\Area}{\text{Area}}
\DeclareMathOperator{\HArea}{\text{HArea}}
\DeclareMathOperator{\FA}{\text{FA}}
\DeclareMathOperator{\Rad}{\text{Rad}}

\DeclareMathOperator{\In}{\text{in}}
\DeclareMathOperator{\out}{\text{out}}

\usepackage{tikz}
\usetikzlibrary{calc}
\usetikzlibrary{automata}
\usetikzlibrary{backgrounds}
\usetikzlibrary{fit}
\usetikzlibrary{arrows}
\usetikzlibrary{arrows.meta}
\usetikzlibrary{decorations.pathreplacing}
\usetikzlibrary{decorations.markings}
\usetikzlibrary{positioning}
\usetikzlibrary{intersections}
\usetikzlibrary{calligraphy}
\usetikzlibrary{patterns}
\usetikzlibrary{shapes}
\usetikzlibrary{shapes.geometric}
\usetikzlibrary{shadows}
\usetikzlibrary{fadings}
\usetikzlibrary{backgrounds}
\usepgfplotslibrary{colorbrewer}
\pgfplotsset{cycle list/Blues-9}
\usetikzlibrary{nfold}
\usepackage{tikz-3dplot}

\tikzset{
    loop above right/.style={above right, out= 60, in= 30, loop},
    loop above left/.style ={above left,  out=150, in=120, loop},
    loop below right/.style={below right, out=330, in=300, loop},
    loop below left/.style ={below left,  out=240, in=210, loop}
}

\tikzset{%
  show curve controls/.style={
        postaction={
            decoration={
                show path construction,
                curveto code={
                \draw [blue] 
                    (\tikzinputsegmentfirst) -- (\tikzinputsegmentsupporta)
                    (\tikzinputsegmentlast) -- (\tikzinputsegmentsupportb);
                \fill [red, opacity=0.5] 
                    (\tikzinputsegmentsupporta) circle [radius=.5ex]
                    (\tikzinputsegmentsupportb) circle [radius=.5ex];
                }
            },
            decorate
        }
    }
}

\title{Homological Isoperimetric Inequalities for Kernels of Free Extensions of Type \texorpdfstring{$FP_2$}{FP_2}}
\date{}
\author{Jakub F. Tucker}
\begin{document}
\begin{abstract}
    We define homological area-radius pairs with surface diagrams. Using these, we adapt a proof of Gersten and Short \cite{gersten2002} to obtain a homological isoperimetric inequality for subgroups of type $FP_2$ which appear as kernels of free extensions.
\end{abstract}

\maketitle
\section{Introduction}
Hyperbolic groups have long been an object of interest in geometric group theory. One characterisation is that a group is hyperbolic if and only if it is finitely presentable and satisfies a linear isoperimetric inequality. In general, there is much less known about subgroups of hyperbolic groups: hyperbolicity is not a property that extends to all subgroups. In \cite{rips1982}, predating Gromov's introduction of the notion of hyperbolic groups, Rips constructs finitely generated subgroups of small cancellation groups which are not finitely presentable. In \cite{brady1999}, Brady constructs the first example of a finitely presented subgroup of a hyperbolic group that is not hyperbolic.

Brady's subgroup appears as the kernel of a surjective map from a hyperbolic group onto $\bZ$. Results of Gromov in \cite{gromov1987} and Bowditch in \cite{bowditch1995} tell us that any group satisfying a subquadratic isoperimetric inequality must also satisfy a linear one, thus the Dehn function of this subgroup is at least quadratic. The question of an upper bound prompted the work of Gersten and Short in \cite{gersten2002}, which concluded:

\begin{thm}{\label{thm:gerhyp}}
     Given an extension
    \begin{center}
        $1\ra K\ra H\ra F_n\ra 1$
    \end{center}
    where $F_n$ is free and $K$ is finitely presented, then if $H$ is hyperbolic, $K$ satisfies a polynomial isoperimetric inequality.
\end{thm}

In \cite{gerstendim2}, Gersten proves that if a hyperbolic group is of cohomological dimension $2$, then all its subgroups of type $FP_2$ are hyperbolic and therefore finitely presented. However, this property is unique to dimension $2$, as shown by Kropholler and Vigolo in \cite{kropholler2021hyperbolicgroupsfinitelypresented}, where a group $K$ which is of type $FP_2$ but not finitely presented (and thus not hyperbolic) appears as the kernel of a homomorphism $\phi:H\ra \bZ$ from a hyperbolic group $H$ of cohomological dimension 3. Notice the similarity to the construction of Brady's subgroup.

There is a notion of isoperimetric inequality for groups $H$ of type $FP_2$, including those that are not finitely presented: the \emph{homological Dehn function} $\FA_H$. Many foundational results are given in \cite{brady2021} by Brady, Kropholler and Soroko. 

The homological Dehn function is defined very similarly to the homotopical Dehn function. As such, we may hope that techniques for bounding traditional Dehn functions can be applied in this new setting, and that the two functions share useful properties. Indeed, there are many parallels with homotopical Dehn functions. For example, Theorem 2.36 of \cite{brady2021} tells us that the homological Dehn function is a quasi-isometry invariant. 

The homological Dehn function differs from its homotopical counterpart in some regards. For one, it is proven in Proposition 2.26 of \cite{brady2021} that every homological Dehn function is $\simeq$-equivalent to a superadditive function, while this is still an open conjecture for homotopical Dehn functions. Significantly, the link between the homological Dehn function and the word problem is weaker than its homotopical counterpart, as discussed in Corollary 6.1 of \cite{brady2021}. There exist groups of type $FP_2$ with unsolvable word problems, whose homological Dehn functions are $\simeq$-equivalent to $n\mapsto n^4$. On the other hand, it is well-known that a finitely presented group $G$ has solvable word problem if and only if its Dehn function $\delta_G$ is bounded above by a recursive function.

One might naturally ask whether the subgroups constructed by Kropholler and Vigolo also satisfy a polynomial upper bound to their homological Dehn function. 

Indeed, such an upper bound is achieved as an immediate corollary of our more general result below, analogous to \cref{thm:gergen} \cite[Theorem B]{gersten2002}:

\textbf{\cref{thm:homger}.}\emph{
    Let $H$ be an extension of a group $K$ of type $FP_2$ by a finitely generated free group $F_n$, so that we have the short exact sequence
    \begin{center}
        $1\ra K\ra H\ra F_n\ra 1$.
    \end{center}
    Then $H$ is of type $FP_2$, and if $(f,g)$ is a homological area-radius pair for $H$ with $f(n)\geq n$ for all $n$, there is a constant $C >1$ such that $\FA_K\preccurlyeq C^gf$.
}

Such $H$ as in the statement are known to be of type $FP_2$ following \cite[Proposition 2.2]{kochloukova2016homologicalfinitenesspropertiesfibre}.

To obtain this result, we adapt the definition of an area-radius pair as used in \cite{gersten2002} by introducing the \emph{homological area-radius pair} as in \cref{def:homarp}. Fundamental properties of usual area-radius pairs also hold for their homological counterpart. Most importantly, shown in \cref{prop:harindep}, the independence of choice of (homological) finite presentation, up to the notion of equivalence of \cref{def:equivs}.

The proof of Gersten and Short's result consists of dividing a disc diagram for a word in $K\le H$ into components existing in $K$, separated by certain annuli called `$t$-rings', which are then `pushed down' by exchanging the disc diagram filling the interior boundary of such annuli, for one filling the exterior boundary. These $t$-rings are pushed down from innermost outward, each time controlling area increase linearly, to obtain a disc diagram entirely within $K$, called a $K$-filling.

The homological perspective introduces new challenges, stemming from working with surface diagrams in place of disc diagrams. Our `$t$-ring' annuli are not necessarily nullhomotopic, prompting us to instead look at `$t$-cycles' as in Subsection \ref{sec:tring}. We also consider the product-like coarse structure of the Cayley graph of the extension $H$ in order to implement `pushing down' one $t$-cycle at a time until we find a $K$-filling.

We conclude with the application of \cref{thm:homger} to subgroups of hyperbolic groups, such as those seen in \cite{kropholler2021hyperbolicgroupsfinitelypresented}, achieving a result analogous to \cref{thm:gerhyp}.

\textbf{\cref{thmB}}\emph{
    Let $H$ be an extension of a group $K$ of type $FP_2$ by a finitely generated free group $F_n$, so that we have the short exact sequence
    \begin{center}
        $1\ra K\ra H\ra F_n\ra 1$.
    \end{center}
    Then if $H$ is hyperbolic, $FA_K$ is bounded above by a polynomial.
}

One should note that in the case of $K$ being finitely presented, both theorems follow the work of Gersten and Short, since \cref{lem:comparedf} bounds the homological Dehn function above by the homotopical version.

\textbf{Outline.} In \cref{se:homdf}, we recall necessary results and definitions regarding homological Dehn functions from \cite{brady2021}. In \cref{sec:gerstenrecap}, we recall the key technical components of Gersten and Short's proof of their Theorem B. In \cref{sec:4}, we adapt the methods of Gersten and Short to prove an analogous result for homological Dehn functions of groups of type $FP_2$, namely \cref{thm:homger}, from which follows \cref{thmB}.

\textbf{Acknowledgements.} The author is grateful to Robert Kropholler, for posing this problem, and many helpful discussions, particularly in regards to \cref{subsec:surfacediagram}; and to Henry Bradford, for comments leading to much-improved presentation of the work. The author is supported by the Engineering and Physical Sciences Research Council [EP/W524633/1].

\section{The Homological Dehn Function}\label{se:homdf}
The homological Dehn function is defined for groups of type $FP_2$ as defined below. Notably, every finitely presented group is of type $FP_2$, but there exist groups of type $FP_2$ which are not finitely presentable, as shown by Bestvina and Brady in \cite{bestvina1997}.

Given a group $H$ with presentation $\pres[A]{R}$, we will not assume that $R$ is finite. We say $H$ is of type $FP_2$ if there exists a \emph{homological finite presentation} for $H$.

\begin{defn}[\cite{brady2021}]
    A homological finite presentation for $H=\pres[A]{R}$, where $A$ is finite, is a pair $\hpres[A]{R_0}$, where $R_0$ is a finite subset of $R$ such that $H_1\brac[X(A,R_0)]=0$. Here $X(A,R_0)$ denotes the homological Cayley complex, a locally finite CW complex constructed similarly to the Cayley $2$-complex. Its 1-skeleton is the Cayley graph $\Gamma(H,A)$. For each vertex $x\in H$ and each relation $r\in R_0$, one attaches a $2$-cell $D_r^x$ with boundary the loop labelled by $r$ beginning at $x$. 
\end{defn}
\begin{rem}
    One should note that more typically in the literature, groups of type $FP_n$ $(n\in\bN)$ are defined in terms of projective resolutions of $\bZ G$-modules. The equivalence of these characterisations for groups of type $FP_2$ is Proposition 2.7 of \cite{brady2021}.
\end{rem}
\begin{rem}
    If $H$ is finitely presented, $\hpres[A]{R}$ is a homological finite presentation for $H$.
\end{rem}
\begin{rem}
    The group $H$ acts freely, cellularly, cocompactly, and vertex transitively on the homological Cayley complex $X(A,R_0)$ on the left as follows. On vertices, an element $h\in H$ maps the vertex $x\in H$ to the vertex $hx$. On edges, an edge from $x$ to $xa$ with label $a$ is mapped by $h$ to the edge from $hx$ to $hxa$ with the same label. On 2-cells, the 2-cell $D_r^x$ is mapped to the 2-cell $D_r^{hx}$. Note that we require $R_0$ to be finite to ensure that this action is cocompact.
\end{rem}

In the homological Cayley complex, every 1-cycle is the boundary of a 2-chain. This property allows us to define the homological filling functions below, and in turn the \emph{homological Dehn function} $\FA_H$ for a group $H$ of type $FP_2$.

\begin{defn}[\cite{brady2021}]
    A map between CW complexes is called \emph{combinatorial} if each open cell of the domain is mapped homeomorphically onto an open cell of the target. A CW complex is called \emph{combinatorial}, or simply a \emph{combinatorial complex}, if all of its attaching maps are combinatorial.
\end{defn}

Importantly, the Cayley $2$-complex of a homological finite presentation $\hpres[A]{R_0}$ is combinatorial, and Proposition 2.6 of \cite{brady2021} asserts a partial converse.
\begin{prop}[\cite{brady2021}]
    Let $X$ be a connected $2$-dimensional combinatorial complex with trivial first homology on which $H$ acts freely, cellularly, cocompactly and vertex-transitively, then there exists a homological finite presentation $\hpres[A]{R_0}$ for $H$ whose homological Cayley complex is $X$.
\end{prop}

\begin{defn}[\cite{brady2021}]
    Let $X$ be a 2-dimensional combinatorial complex with $H_1(X)=0$. Let $\gamma$ be a 1-cycle in $X^{(1)}$, so that $\gamma$ can be realised as a combinatorial map $\bigsqcup_iS^1\ra X$ (when $\gamma$ is a loop this map is $S^1\ra X$) and $|\gamma|$ is defined to be the minimal number $\sum_j|b_j|$ of $1$-cells in $\bigsqcup_iS^1$ among all representations $\sum_jb_je_j$ of $\gamma:\bigsqcup_iS^1\ra X$ where $b_j\in\bZ$ and $e_j$ are edges in $X$ for finitely many $j$.

    We define the \emph{area} $\Area(c)$ of a 2-chain $c=\sum_i a_i\sigma_i$ to be $\sum_i|a_i|$ and the \emph{homological area} of $\gamma$ to be
        $$\HArea_X(\gamma)\coloneqq \min\left\{\Area(c)\,\middle|\, \gamma=\del c \text{ for }c=\sum_ia_i\sigma_i\right\}.$$
    The \emph{homological filling function} is
        $$\FA_X(n)\coloneqq \sup\setst[\HArea_X(\gamma)]{\gamma\text{ is a $1$-cycle in $X^{(1)}$ with }|\gamma|\leq n}.$$

    A \textit{homological isoperimetric inequality for }$X$ is a function $f:\nton$ with the property that $\FA_X(n)\le f(n)$ for all $n$.
\end{defn}

\begin{defn}\label{def:equivs}
    For functions $f,\ol[f]:\bN\ra\bR$, we say $f\preccurlyeq \ol[f]$ if there exists a constant $C$ such that $f(n)\leq C\ol[f](Cn+C)+Cn+C$, and $f\simeq f$ if both $f\preccurlyeq \ol[f]$ and $\ol[f]\preccurlyeq f$. For functions $g,\ol[g]:\bN\ra\bR$, we say $g\cong \ol[g]$ if there exists a constant $C$ such that both $g(n)\leq C\ol[g](Cn)+C$ and $\ol[g](n)\leq Cg(Cn)+C$ for all $n\in\bN$.
\end{defn}
\begin{rem}
    Note that `$\cong$' is a strictly stronger relation than `$\simeq$'.
\end{rem}

\begin{defn}[\cite{brady2021}]
    Consider a group $H$ of type $FP_2$ and a homological finite presentation \hpres[A]{R_0}, the \emph{homological Dehn function} of the triple $\langle H,A,R_0\rangle$ is
        $$\FA_{\langle H,A,R_0\rangle}(n)\coloneqq \FA_{X(A,R_0)}(n).$$
    We call a homological isoperimetric inequality for the homological Cayley complex $X(A,R_0)$ a \emph{homological isoperimetric inequality for }$\hpres[A]{R_0}$.
\end{defn}

Like the homotopical Dehn function, Proposition 2.24 of \cite{brady2021} tells us that, up to $\simeq$-equivalence, the homological Dehn function is independent of choice of homological finite presentation. So we can write $\FA_H$ without specifying a homological finite presentation. We can also define a \emph{homological isoperimetric inequality for }$H$ to be a function $f:\nton$ with the property that $\FA_H\preccurlyeq f$.



When a group $H$ of type $FP_2$ is also finitely presented, we may compare the homological and homotopical Dehn functions. 
\begin{lem}[\cite{brady2021}, Prop 2.28]\label{lem:comparedf}
    If $H\cong\pres[A]{R}$ is a finitely presented group, then $\FA_H(n)\preccurlyeq\overline{\delta_H}(n)$.
\end{lem}
Here $\ol[\delta_H]$ is the \emph{superadditive closure} of the Dehn function $\delta_H$:
    $$\ol[f](n)\coloneqq\max\setst[f(n_1)+\cdots+f(n_r)]{r\geq 1, n_i\in\bN, n_1+\cdots+n_r=n}.$$

\section{Finitely Presented Kernels of Free Extensions}\label{sec:gerstenrecap}
Given a group $H$ with finite presentation $\mathcal{P}=\pres[A]{R}$: 
$F(A)$ denotes the free group on $A$, $\ol[A]$ denotes the set $A\cup A\inv$ and $\ol[A]^*$ denotes the free monoid on $\ol[A]$.

We consider words $w\in H$ with $w=_H 1$ as elements of $F(A)$ or $\ol[A]^*$ and their Van Kampen diagrams over $\mathcal{P}$ ($\mathcal{P}$-diagrams) as connected, oriented, labelled, planar graphs $D$, with an associated combinatorial map $\pi:D\ra \Gamma(H,A)$ which preserves orientations and labels of edges. Embedding a $\mathcal{P}$-diagram $D$ in the plane, the bounded regions of the complement of $D$ have boundaries labelled by elements of $R$, and the unbounded region has boundary labelled by $w$, the number of bounded regions is the area $\Area(D)$ of $D$. For $w$ as above, the \emph{area} of $w$ is 
$$\Area_{\mathcal{P}}(w)=\min\left\{\Area(D)\middle| (D,\pi)\text{ is a }\mathcal{P}\text{-diagram whose unbounded region has label }w\right\}.$$
An \emph{area function} $f:\bN\ra\bR$ for $\mathcal{P}$ is a function such that for every $n\in\bN$, and every word $w=_H1$ in $F(A)$ with length at most $n$, $\Area_{\mathcal{P}}(w)\leq f(n)$. Such a function is also called an \emph{isoperimetric inequality} and the Dehn function $\delta_\mathcal{P}$ is a minimal such function. A \emph{radius function} $g:\bN\ra\bR$ for $\mathcal{P}$ is a function such that for each word $w=_H1$ in $F(A)$, there is a Van Kampen diagram $D_\mathcal{P}(w)$ such that for every vertex $v$ in $D_\mathcal{P}(w)$, there is a path in the $1$-skeleton from $v$ to the boundary $\del D_\mathcal{P}(w)$ of length at most $g(n)$. A pair of such functions is called an \emph{area-radius} (AR) pair.

We say that AR pairs $(f,g)$ and ($\ol[f],\ol[g])$ are \textit{equivalent} when $f\simeq \ol[f]$ and $g\cong \ol[g]$. Proposition 2.1 of \cite{gersten2002} tells us that the equivalence class of an AR pair is a group invariant, \textit{i.e.} if $\mathcal{P}$ and $\mathcal{Q}$ are finite presentations for $H$, then there are equivalent AR pairs $(f,g)$ for $\mathcal{P}$ and $(\ol[f],\ol[g])$ for $\mathcal{Q}$.

Gersten and Short's general result, Theorem B of \cite{gersten2002}, which implies the specific case \cref{thm:gerhyp}, is as follows.
\begin{thm}{\label{thm:gergen}}
    Given an extension
    \begin{center}
        $1\ra K\ra H\ra F_n\ra 1$
    \end{center}
    where $F_n$ is free and $K$ is finitely presented, if $(f,g)$ is an area-radius pair for $H$ with $f(n)\geq n$ for all $n\in\bN$, then there exists a constant $C>1$ such that $C^gf$ is an isoperimetric inequality for $K$.
\end{thm}

To prove Gersten and Short's result, it is necessary to understand how the area of a Van Kampen diagram can be affected by the application of an automorphism, which leads to Lemmas 3.1 and 3.2 of \cite{gersten2002} (our Lemmas \ref{lem:areacap1}, \ref{lem:areacap2}):
\begin{notation}\label{not:bigphi}
    Let $\mathcal{P}=\pres[A]{R}$ be a presentation for a group $H$, where $A$ is finite, and let $\phi:H\ra H$ be a group automorphism. For each $a\in A$, choosing a word representing $\phi(a)$ in $F(A)$ induces a monoid homomorphism $\Phi:\ol[A]^*\ra\ol[A]^*$ with $\Phi(a)=_H\phi(a), \Phi(a\inv)=_H\phi(a\inv)=\phi(a)\inv$. Since $\phi$ is an automorphism, its inverse also induces a monoid homomorphism $\Psi:\ol[A]^*\ra\ol[A]^*$ where $\Phi(\Psi(a^{\pm 1}))=_Ha^{\pm 1}=_H\Psi(\Phi(a^{\pm1}))$.
\end{notation}
Note that Gersten and Short assume that the presentation $\mathcal{P}$ is finite.
\begin{lem}{\label{lem:areacap1}}
    There is a constant $S>0$ such that if $D_\mathcal{P}(w)$ is a Van Kampen diagram for a word $w=_H1$, then there is a Van Kampen diagram $D'_\mathcal{P}(\Phi(w))$ over $\mathcal{P}$ for $\Phi(w)$ with 
    \begin{center}
        \emph{$\Area\left(D'_\mathcal{P}(w)\right)\leq S\cdot \Area\left(D_\mathcal{P}(w)\right)$},
    \end{center}
\end{lem}
and conversely,
\begin{lem}{\label{lem:areacap2}}
    There are constants $S',S''>0$ such that if $D'_\mathcal{P}(\Phi(w))$ is a Van Kampen diagram for $\Phi(w)$, then there is a Van Kampen diagram $D''_\mathcal{P}(w)$ for $w$ such that
    \begin{center}
        \emph{$\Area\left(D''_\mathcal{P}(\Phi(w))\right)\leq S'\cdot \Area\left(D'_\mathcal{P}(\Phi(w))\right)+S''\cdot \ell(w)$.}
    \end{center}
\end{lem}

The key observation for Gersten and Short's theorem is as follows. Given a word $w=_K1$, we consider a Van Kampen diagram $D$ over $\mathcal{P}_H$. Since there are no $t$-edges on the boundary, \cite{gersten1996} tells us each edge labelled $t$ is in a unique annulus formed of $2$-cells corresponding to relations $t_i\inv x_jt_i\Phi(x_j)$ which lies in the interior of $D$, called a `$t$-ring'. Gersten and Short use Lemmas \ref{lem:areacap1}, \ref{lem:areacap2} to exchange a disc diagram filling for the interior boundary of a $t$-ring with one for the exterior boundary, while controlling linearly the increase in area.

\section{Kernels of type \texorpdfstring{$FP_2$}{FP2}}\label{sec:4}
In this section, we work under the hypothesis that we have an extension
\begin{center}
    $1\ra K\ra H\ra F_n \ra 1$,
\end{center}
where the groups $K$ and $H$ are of type $FP_2$, and $F_n$ is free of rank $n$. We give homological analogues of key results from \cite{gersten2002} used in the proof of \cref{thm:gergen}, beginning with expanding the notion of an area-radius pair on a disc diagram, adapting the results of \cref{lem:areacap1} and \cref{lem:areacap2}, before fitting the results together as Gersten and Short did.

\subsection{Surface Diagrams}\label{subsec:surfacediagram}
In \cite{gersten2002}, Gersten and Short do not work directly in the Cayley $2$-complex for $H$, but rather on \emph{disc diagrams}, $2$-discs with an associated combinatorial structure that map into the Cayley $2$-complex as a filling. We use an analogous approach and consider a \emph{surface diagram}, a space with simpler geometry than an arbitrary $2$-chain realised in the homological Cayley complex, while still obtaining the same information about the group.

To construct explicitly these surface diagrams, we will take cues from Arenas and Wise's definition in \cite{arenas2022linearisoperimetricfunctionssurfaces}, as well as the definitions of (combinatorial) disc and annular diagrams and their filling maps in \cite{brady2021}.

\begin{defn}[Surface Diagram]
    A \emph{surface diagram} $S$ is a compact combinatorial $2$-complex defined as 
    \begin{center}
        $\ol[S]\setminus \left(\bigsqcup_{i=1}^n e_i\right)$,
    \end{center}
    for some closed, possibly disconnected, surface $\ol[S]$, where the $e_i$ are finitely many disjoint open $2$-cells in a combinatorial cell structure on $\ol[S]$. The surface diagram has $n$ \emph{boundary paths}, corresponding to the attaching maps $\phi_i:\del D^2_{e_i}\hookrightarrow \ol[S]$. The \emph{boundary} $\del S $ of $S$ is the union of such paths.

    The \emph{area} $\Area (S)$ of a surface diagram $S$ is the number of $2$-cells in $S$. The \emph{radius} $\Rad (S)$ of a surface diagram $S$ is $\max_{v\in S^{(0)}}\left(\min\left\{|\gamma|\,\middle|\,\gamma\text{ is a path from $v$ to $\del S$}\right\}\right)$.
\end{defn}

\begin{defn}[Surface Diagram filling of a $1$-cycle]
    Let $X$ be a combinatorial complex with a $1$-cycle $\gamma$ in the $1$-skeleton $X^{(1)}$. A \emph{filling} for $\gamma$ is a pair $(S,\pi)$ consisting of a surface diagram $S$ with a combinatorial map $\pi:S\ra X$ where $\gamma=\pi\circ\bigsqcup_{i=1}^n \phi_i$, and each $\pi\circ\phi_i\subseteq \gamma$ is a closed loop in $X^{(1)}$. \emph{i.e.} each boundary loop maps to a loop in the $1$-cycle $\gamma$.
\end{defn}

The following result is known, but was not found in the literature. As such, we have included a full proof below.

\begin{lem}\label{lem:sdexists}
    Let $X=X(A,R_0)$ be a a homological Cayley complex for a group $H$ of type $FP_2$ with homological finite presentation $\hpres[A]{R_0}$, let $\gamma$ be a $1$-cycle in the $1$-skeleton $X^{(1)}$. Then there exists a surface diagram filling $(S,\pi)$ for $\gamma$.
\end{lem}

\begin{proof}
    The complex $X$ is a homological Cayley complex so $H_1(X)=0$ and there exists a (finite) $2$-chain $c=\sum_{i=1}^ka_i\sigma_i$ in $X$ with boundary $\del c=\gamma$. Each $2$-cell $\sigma_i$ has boundary a loop in $X^{(1)}$ labelled by a relator $r_i\in R_0$ beginning at a vertex $h_i\in H$.

    We begin with a collection of all the $2$-cells in $c$, with multiplicity given by the $a_i$ in the sum. We construct a combinatorial structure of a several-times punctured surface $S$ which maps combinatorially to $c$ as follows. Pick some $2$-cell $\sigma_i$ from this collection to add to our surface diagram $S$, consider its boundary word $r_i$ starting at $g_i$ in $c$. We say an edge in the loop $r_i$ starting at $g_i$ is \emph{shared with} $\sigma_j$ if it is contained in the boundary of another $2$-cell $\sigma_j$, we call the edge \emph{gluable} if it is either shared with some $\sigma_j$ in the collection, or shared with a $2$-cell $\sigma_k$ already in $S$, with the corresponding edge lying on the boundary $\del S$. If an unglued edge does not appear in the boundary of $c$, then it must appear in $c$ an even number of times so that its value is zero when the boundary is taken, so it must be gluable.
    
    If $\sigma_i$ has no gluable edges, then its boundary must form a loop in $\gamma$ so this $2$-cell forms its own connected component of $S$ and we move on to a new $2$-cell in $c$. If $\sigma_i$ has a gluable edge with label $a\in A$ and initial vertex $h$, pick some other $2$-cell $\sigma_j$ sharing this edge in $c$, and glue along it. 
    
    Return to the vertex $h$ in $S$, there are now 2 unglued edges with $h$ as their initial vertex, if neither is gluable, then we move on to another vertex in the boundary of $S$. If one of the edges containing $h$ in $S$ is gluable, then pick another $2$-cell $\sigma_k$ in $c$ that shares this edge, and glue along it. Note that we might choose a $\sigma_k$ that is already in $S$, thus identifying in $S$ two (possibly previously separate) vertices corresponding to the same group element. There might be other $2$-cells that share this edge in $c$, but in the surface diagram $S$ we do not necessarily identify their endpoints labelled $h$ in $c$. Compare with \cref{fig:surface2}.

    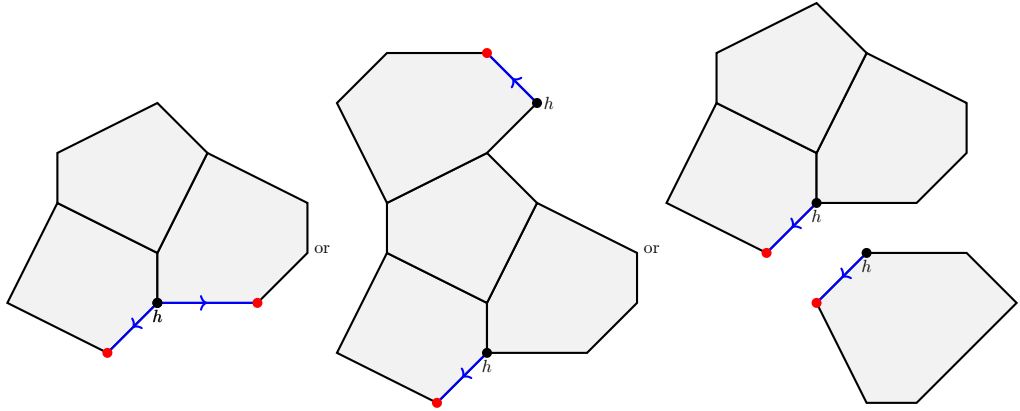
\begin{figure}[!ht]
    \centering
    \resizebox{0.9\textwidth}{!}{
    \begin{tikzpicture}[very thick, fill=black!5!white, baseline={(0,3)},decoration={
    markings,
    mark=at position 0.5 with {\arrow{>}}}]
        \filldraw (1,4) -- ++(2,-1) -- ++(0,-1) -- ++(-1,-1) -- ++(-2,1) -- cycle;
        \filldraw (3,3) -- ++(1,2) -- ++(2,-1) -- ++(0,-1) -- ++(-1,-1) -- ++(-2,0) -- node [pos=0, below] {$h$} cycle;
        \filldraw (1,4) -- ++(0,1) -- ++(2,1) -- ++(1,-1) -- ++(-1,-2) -- cycle;
        \draw[blue,postaction=decorate] (3,2) -- (2,1);
        \fill[black] (3,2) circle (0.1);
        \fill[red] (2,1) circle (0.1);
        \draw[blue,postaction=decorate] (3,2) -- (5,2);
        \fill[black] (3,2) circle (0.1) node [below] {$h$};
        \fill[red] (5,2) circle (0.1);
    \end{tikzpicture}
    or
    \begin{tikzpicture}[very thick, fill=black!5!white, baseline={(0,4)},decoration={
    markings,
    mark=at position 0.5 with {\arrow{>}}}]
        \filldraw (1,4) -- ++(2,-1) -- ++(0,-1) -- ++(-1,-1) -- ++(-2,1) -- cycle;
        \filldraw (3,3) -- ++(1,2) -- ++(2,-1) -- ++(0,-1) -- ++(-1,-1) -- ++(-2,0) -- node [pos=0, below] {$h$} cycle;
        \filldraw (1,4) -- ++(0,1) -- ++(2,1) -- ++(1,-1) -- ++(-1,-2) -- cycle;
        \filldraw (1,5) -- ++(-1,2) -- ++(1,1) -- ++(2,0) -- ++(1,-1) -- ++(-1,-1) -- cycle;
        \draw[blue,postaction=decorate] (3,2) -- (2,1);
        \draw[blue,postaction=decorate] (4,7) -- (3,8);
        \fill[black] (3,2) circle (0.1);
        \fill[black] (4,7) circle (0.1) node [right] {$h$};
        \fill[red] (2,1) circle (0.1);
        \fill[red] (3,8) circle (0.1);
    \end{tikzpicture}
    or
    \begin{tikzpicture}[very thick, fill=black!5!white, baseline={(0,1)},decoration={
    markings,
    mark=at position 0.5 with {\arrow{>}}}]
        \filldraw (1,4) -- ++(2,-1) -- ++(0,-1) -- ++(-1,-1) -- ++(-2,1) -- cycle;
        \filldraw (3,3) -- ++(1,2) -- ++(2,-1) -- ++(0,-1) -- ++(-1,-1) -- ++(-2,0) -- node [pos=0, below] {$h$} cycle;
        \filldraw (1,4) -- ++(0,1) -- ++(2,1) -- ++(1,-1) -- ++(-1,-2) -- cycle;
        \filldraw (3,0) -- ++(1,1) -- ++(2,0) -- ++(1,-1) -- ++(-2,-2) -- ++(-1,0) -- cycle;
        \draw[blue,postaction=decorate] (3,2) -- (2,1);
        \fill[black] (3,2) circle (0.1);
        \fill[red] (2,1) circle (0.1);
        \draw[blue,postaction=decorate] (4,1) -- (3,0);
        \fill[black] (4,1) circle (0.1) node [below] {$h$};
        \fill[red] (3,0) circle (0.1);
    \end{tikzpicture}
    }
    \caption{Gluing edges around $h$, highlighted edges may be glued with the orientation shown.}
    \label{fig:surface2}
    \end{figure}

    We now return to this vertex $h$ in $S$ and repeat the procedure with any gluable edges with $h$ at one end, attaching $2$-cells until the only unglued edges aout $h$ correspond to those in $\gamma$, or there are no unglued edges. As in \cref{fig:surface1}: in the former case, locally around $h$ we have a neighbourhood homeomorphic to the upper half plane in $\bR^2$; in the latter, we have a neighbourhood homeomorphic to $\bR^2$. Notice that any point in an internal edge we have glued is locally homeomorphic to $\bR^2$, and any point on a boundary edge is locally homeomorphic to the upper half plane, and any point internal to a $2$-cell is locally homeomorphic to $\bR^2$. Once this vertex $h$ is locally a surface -- possibly with boundary -- we turn our attention to a different vertex in $S$, and repeat the process outlined above.

    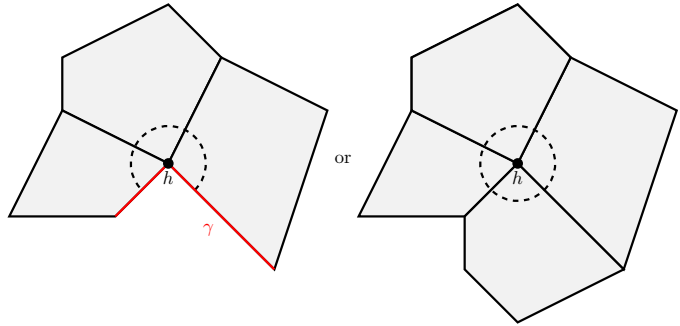
\begin{figure}[!ht]
    \centering
    \resizebox{0.6\textwidth}{!}{
    \begin{tikzpicture}[very thick, fill=black!5!white, baseline={(0,3)}]
        \filldraw (1,5) -- ++(2,1) -- ++(1,-1) -- ++(-1,-2) -- node [pos=0, below] {$h$} ++(-2,1) -- cycle;
        \filldraw (1,4) -- ++(2,-1) -- ++(-1,-1) -- ++(-2,0) -- cycle;
        \filldraw (3,3) -- ++(1,2) -- ++(2,-1) -- ++(-1,-3) -- cycle;
        \draw[dashed] (3.5,2.5) arc [start angle=-45, delta angle=270, radius = 0.707];
        \draw[red] (2,2) -- ++(1,1) -- node [pos=0.5, below left] {$\gamma$} ++(2,-2);
        \fill[black] (3,3) circle (0.1);
    \end{tikzpicture}
    or
    \begin{tikzpicture}[very thick, fill=black!5!white, baseline={(0,3)}]
        \filldraw (1,5) -- ++(2,1) -- ++(1,-1) -- ++(-1,-2) -- node [pos=0, below] {$h$} ++(-2,1) -- cycle;
        \filldraw (1,4) -- ++(2,-1) -- ++(-1,-1) -- ++(-2,0) -- cycle;
        \filldraw (3,3) -- ++(1,2) -- ++(2,-1) -- ++(-1,-3) -- cycle;
        \filldraw (2,2) -- ++(1,1) -- ++(2,-2) -- ++(-2,-1) -- ++(-1,1) -- cycle;
        \filldraw (1,5) -- ++(2,1) -- ++(1,-1) -- ++(-1,-2) -- node [pos=0, below] {$h$} ++(-2,1) -- cycle;
        \draw[dashed] (3.5,2.5) arc [start angle=-45, delta angle=360, radius = 0.707];
        \fill[black] (3,3) circle (0.1);
    \end{tikzpicture}
    }
    \caption{Near $h$ is a surface (with boundary).}
    \label{fig:surface1}
    \end{figure}

    If this process appears to terminate early, that is, we have a closed surface with boundary $S$, but there are still $2$-cells in $c$ unaccounted for in $S$, then these comprise a separate connected component in $S$. We can continue to build our surface diagram by picking any unaccounted-for $2$-cell in $c$ and repeating the construction.

    Since the $2$-chain $c$ is finite, we have only finitely many edges, vertices and $2$-cells, so this process must terminate. Note that it is impossible for the process to terminate without yielding a (possibly disconnected) surface with boundary, since any point internal to a $2$-cell has a neighbourhood homeomorphic to $\bR^2$, as does any point on a glued edge, and each vertex has a neighbourhood homeomorphic to either $\bR^2$ or the upper half plane, as do the unglued edges. Moreover, every unglued edge must correspond to an edge on the boundary of $c$, else it would have been gluable. Thus, once this construction terminates, every point of $S$ has a neighbourhood homeomorphic to $\bR^2$ or the upper half plane, that is, $S$ is a \emph{surface with boundary}.

    We have constructed the surface diagram $S$ in such a way that the filling map is clear. We define $\pi:S\ra c\subseteq X$ to be the map that sends each point in $S$ that is a vertex to the corresponding vertex in $c$, every point on an edge in $S$ to the corresponding point in the corresponding edge in $c$, and every point in a 2-cell to the corresponding point in the corresponding 2-cell in $c$. By construction, $n$-cells are mapped homeomorphically onto $n$-cells ($\pi$ restricts to the identity map on any given $n$-cell), and thus the map is combinatorial. The boundary components of $S$ correspond exactly to $\gamma$ by construction, so $\gamma=\pi(\del S)$ and $\pi$ is a surface diagram filling for $\gamma$. 
\end{proof}

\begin{rem}\label{rem:samearea}
    By construction, the area of the surface diagram $S$ constructed in the proof above is precisely the area of the $2$-chain that is its image -- they have the same number of $2$-cells.
\end{rem}

\subsection{Homological Area-Radius Pairs}
Gersten and Short's upper bound is in terms of an area-radius pair. A priori, these are defined only for finitely presented groups.
Consider the definition of the area-radius pair in \cref{sec:gerstenrecap}, and observe that it still makes sense when using a surface diagram in place of a Van Kampen diagram.

\begin{defn}\label{def:homarp}
    Given a group $H$ of type $FP_2$ with homological finite presentation $\mathcal{P_0}=\hpres[A]{R_0}$, we call the pair $(f,g)$ of functions $f,g:\bN\ra\bR$ a \emph{homological area-radius pair} if for any $1$-cycle $\gamma$ of length at most $n$ in $X(A,R_0)$, there exists a surface diagram filling $(S,\pi)$ for $\gamma$ and $\Area(S)\leq f(n)$ and $\Rad(S)\leq g(n)$.
\end{defn}

\begin{rem}
    If $H$ is a finitely presented group, then every Van Kampen diagram is a homological filling, so the usual area-radius pair is a homological area-radius pair for any fixed finite presentation.
\end{rem}
\begin{rem}
    Like Gersten and Short, we consider two homological area-radius pairs $(f,g), (\ol[f],\ol[g])$ to be equivalent if $f\simeq \ol[f]$ and $g\cong\ol[g]$.
\end{rem}

It remains to check that the homological area-radius pair is well-defined. We know from Proposition 2.24 of \cite{brady2021} that if $f$ is an isoperimetric inequality for a given homological finite presentation $\hpres[A]{R_0}$, then $f$ is $\simeq$-equivalent to an isoperimetric inequality for any other presentation $\hpres[A_1]{S_0}$. It remains to check that the analogous property holds for the radius function.

\begin{prop}\label{prop:harindep}
    Let $\hpres[A]{R_0},\hpres[A_1]{S_0}$ be two homological finite presentations for a group $H$ of type $FP_2$, and $(f,g)$ a homological area-radius pair for $\hpres[A]{R_0}$. Then there exists a homological area-radius pair $(\ol[f],\ol[g])$ for $\hpres[A_1]{S_0}$, with $f\simeq \ol[f]$ and $g\cong\ol[g]$.
\end{prop}
The proof below follows the structure of Proposition 2.24 of \cite{brady2021}, with homological AR pairs in place of the homological Dehn function.
\begin{proof}
    We first prove the simpler case where $A=A_1$.
    
    Let $\mathcal{P}_0=\hpres[A]{R_0}, \mathcal{P}_1=\hpres[A]{R_1}$ be homological finite presentations for $H$. Consider the homological Cayley complex $X_0$ for $\mathcal{P}_0$. Note that this space shares its $1$-skeleton with the homological Cayley complex $X_1$ for $\mathcal{P}_1$: both are the Cayley graph for $H$ with respect to generating set $A$.
    
    Let $\gamma$ be a loop of length at most $n$ in $X_0^{(1)}$ with surface diagram filling $(S,\pi)$ with a radius $r\le g(n)$ and area $a\le f(n)$. The boundary of each 2-cell $\del\sigma_i$ in $S$ is a loop labelled by $r_i\in R_0$, so can be filled by a surface diagram in $\mathcal{P}_1$ with radius $n_i$ and area $a_i$. Filling all the $\sigma_i$ in this way, and gluing all these surfaces together corresponding to the gluings in $S$ gives a surface diagram filling $S'$ for $\gamma$ with edge labels from $\mathcal{P}_1$. The radius of $S'$ is bounded above by
        $g(n) + \max_{r_i\in R_0}\{n_i\}$
    since for any vertex $v$ in this diagram, the distance in the $1$-skeleton from $v$ to the boundary $\del\sigma_i$ is at most $\max_i\{n_i\}$, and the distance from that vertex to the boundary $\del S'=\del S$ is at most $g(n)$. By construction, the area of $S'$ is bounded above by $\max_{r_i\in R_0}\{a_i\}\cdot f(n)$. Since $\max_i\{n_i\},\max_i\{a_i\}$ are constants, we have that $(\ol[f],\ol[g])$, defined by $\ol[f](n)=\max_{r_i\in R_0}\{a_i\}\cdot f(n),\ol[g](n)=g(n)+\max_{r_i\in R_0}\{n_i\}$ is a homological AR pair for $\mathcal{P}_1$ with $f\simeq \ol[f]$ and $g\cong \ol[g]$.

    \begin{figure}
        \centering
        \resizebox{0.7\textwidth}{!}{
        \begin{tikzpicture}[very thick, fill=black!5!white, baseline={(0,4)},decoration={
        markings,
        mark=at position 0.5 with {\arrow{>}}}]
            \filldraw (1,4) -- ++(2,-1) -- ++(0,-1) -- ++(-1,-1) -- ++(-2,1) -- cycle;
            \filldraw (3,3) -- ++(1,2) -- ++(2,-1) -- ++(0,-1) -- ++(-1,-1) -- ++(-2,0) -- cycle;
            \filldraw (1,4) -- ++(0,1) -- ++(2,1) -- ++(1,-1) -- ++(-1,-2) -- cycle;
            \filldraw (1,5) -- ++(-1,2) -- ++(1,1) -- ++(2,0) -- ++(1,-1) -- ++(-1,-1) -- cycle;
            \filldraw (4,7) -- ++(2,0) -- ++(1,-1) -- (6,4) -- (4,5) -- (3,6) -- cycle;
            \filldraw (0,2) -- (1,4) -- (1,5) -- (0,7)-- (-1,6) -- (-1,4) -- cycle;
        \end{tikzpicture}
        $\leadsto$
        \begin{tikzpicture}[very thick, fill=black!5!white, baseline={(0,4)},decoration={
        markings,
        mark=at position 0.5 with {\arrow{>}}}]
            \filldraw (1,4) -- ++(2,-1) -- ++(0,-1) -- ++(-1,-1) -- ++(-2,1) -- cycle;
            \filldraw (3,3) -- ++(1,2) -- ++(2,-1) -- ++(0,-1) -- ++(-1,-1) -- ++(-2,0) -- cycle;
            \filldraw (1,4) -- ++(0,1) -- ++(2,1) -- ++(1,-1) -- ++(-1,-2) -- cycle;
            \filldraw (1,5) -- ++(-1,2) -- ++(1,1) -- ++(2,0) -- ++(1,-1) -- ++(-1,-1) -- cycle;
            \filldraw (4,7) -- ++(2,0) -- ++(1,-1) -- (6,4) -- (4,5) -- (3,6) -- cycle;
            \filldraw (0,2) -- (1,4) -- (1,5) -- (0,7)-- (-1,6) -- (-1,4) -- cycle;
            
            \draw[color=black!30!white] (2.5,4.5) -- (3,6);
            \draw[color=black!30!white] (2.5,4.5) -- (4,5);
            \draw[color=black!30!white] (2.5,4.5) -- (1,5);
            \draw[color=black!30!white] (2.5,4.5) -- (3,3);
            
            \draw[color=black!30!white] (2,1) -- (1.5,3) -- (1,4);
            \draw[color=black!30!white] (3,2) -- (0.75,2.5) -- (1.5,3);
            \draw[color=black!30!white] (0.75,2.5) -- (0,2);

            \draw[color=black!30!white] (5,2) -- (4.5,3) -- (6,3) -- (4,5);
            \draw[color=black!30!white] (3,2) -- (4.5,3) -- (3,3);

            \draw[color=black!30!white] (7,6) -- (4,5) -- (4,7) -- cycle;

            \draw[color=black!30!white] (1,5) -- (1,8) -- (4,7);

            \draw[color=black!30!white] (1,5) -- (-1,6) -- (1,4) -- (-1,4) -- cycle;

            \draw (1,4) -- ++(2,-1) -- ++(0,-1) -- ++(-1,-1) -- ++(-2,1) -- cycle;
            \draw (3,3) -- ++(1,2) -- ++(2,-1) -- ++(0,-1) -- ++(-1,-1) -- ++(-2,0) -- cycle;
            \draw (1,4) -- ++(0,1) -- ++(2,1) -- ++(1,-1) -- ++(-1,-2) -- cycle;
            \draw (1,5) -- ++(-1,2) -- ++(1,1) -- ++(2,0) -- ++(1,-1) -- ++(-1,-1) -- cycle;
            \draw (4,7) -- ++(2,0) -- ++(1,-1) -- (6,4) -- (4,5) -- (3,6) -- cycle;
            \draw (0,2) -- (1,4) -- (1,5) -- (0,7)-- (-1,6) -- (-1,4) -- cycle;
        \end{tikzpicture}
        }
        \caption{Obtaining the surface diagram $S'$ by refining a surface diagram $S$.}
        \label{fig:newradius}
    \end{figure}
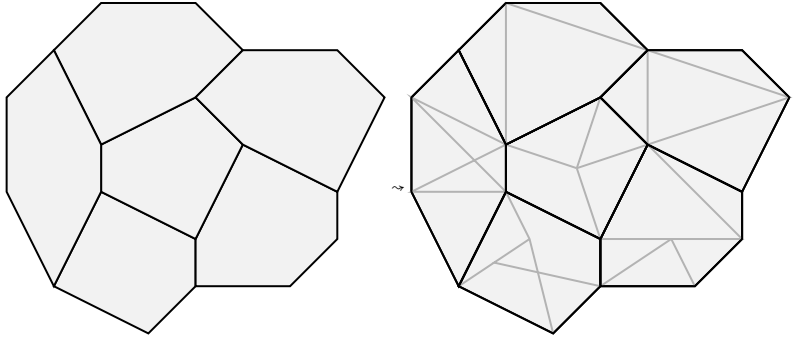
    
    This brings us to the general case. Suppose now that $\hpres[A]{R_0}$ and $\hpres[A_1]{S_0}$ are two homological finite presentations for $H$. For each generator $b$ in $A_1$, there is a word $v_b(A)$ in $F(A)$ with $b=_Hv_b(A)$. Fix a word $v_b(A)$ for each $b\in A_1$ and let $R_1$ be the set of relators $\left\{b\inv v_b(A)\,|\,b\in A_1\right\}\subseteq F(A\cup A_1)$. As in \cite{brady2021}, we construct a new complex $X_1$ from the homological Cayley complex $X(A,R_0)$ as follows: for each generator $b\in A_1$ and for each vertex $x\in X(A,R_0)$, we attach an edge $e_b^x$ labelled by $b$ with initial vertex $x$ and terminal vertex $x\cdot v_b(A)$; then for each such edge, we attach a $2$-cell $D_b^x$ to the closed path reading $b\inv v_b(A)$ from $x$. So the complex $X_1$ is
    $$X_1\coloneqq X(A,R_0)\cup \bigcup_{\substack{{x\in H}\\{b\in A_1}}}e_b^x\cup \bigcup_{\substack{{x\in H}\\{b\in A_1}}}D_b^x$$
    The $H$-action on $X=X(A,R_0)$ is extended to a free, vertex transitive and cocompact action on this new complex by defining $g\cdot e_b^x=e_b^{g\cdot x}$ and $g\cdot D_b^x=D_b^{g\cdot x}$. As noted in the proof of Proposition 2.24 of \cite{brady2021}, we have a deformation retraction $r:X_1\ra X$ which sends edges labelled $b$ to paths labelled $v_b(A)$ with the same initial and terminal vertices, and collapses the $2$-cells $D_b^x$. The existence of this deformation retraction tells us $H_1(X_1)\cong H_1(X(A,R_0))$ is trivial, so $\hpres[A\sqcup A_1]{R_0\sqcup R_1}$ is a homological finite presentations for $H$. One analogously constructs a homological Cayley complex $X_2=X(A\sqcup A_1,S_0\sqcup S_1)$, where $S_1=\{a\inv u_a(A_1)|a\in A)\}$. Let $k=\max\{v_b(A)\,|\,b\in A_1\}\cup\{u_a(A_1)\,|\,a\in A\}$.

    Consider a 1-cycle $\gamma$ of length $n$ in $X_1$. Then $r(\gamma)$ is a 1-cycle of length at most $kn$ in $X$. Let $S$ be a surface diagram for $r(\gamma)$ with $\Area(S)\le f(kn)$ and $\Rad(S)\le g(kn)$. Every edge with label $b\in A_1$ in $\gamma$ corresponds to a subarc of $r(\gamma)$ with label $v_b(A)$. For each such subarc in $S$, glue a $2$-cell with boundary label $b\inv v_b(A)$. This yields a surface $S'$ with $\Area(S')\le \Area(S)+n\le f(kn)+n$ and $\Rad(S')\le \Rad(S)+k\le g(kn)+k$. We can extend $\pi$ to a map $\pi':S'\ra X_1$ by mapping the new 2-cells to the ones collapsed when the deformation retraction $r$ was applied to $\gamma$. This yields a surface diagram for $\gamma$, and we deduce that $(f_1,g_1)$ is a homological AR pair for $\hpres[A\sqcup A_1]{R_0\sqcup R_1}$, where $f_1(n)\coloneqq f(kn)+n$ and $g_1(n)\coloneqq g(kn)+k$, so $f\simeq f_1$ and $g\cong g_1$.

    By the first part of the proof, there exists a homological AR pair $(f_2,g_2)$ for $X_2$ with $f_1\simeq f_2$ and $g_1\cong g_2$. Now let $\gamma$ be a 1-cycle of length $n$ in $\ol[X]=X(A_1,S_0)$. View $\gamma$ as a 1-cycle in $X_2$, let $(S,\pi)$ be a surface diagram filling for $\gamma$ with $\Area(S)\le f(n)$, $\Rad(S)\le g(n)$. Note that no edges with label $a\in A$ lie on $\gamma=\del S$. Suppose that such an $A$-edge appears in the interior. The only 2-cells with boundary label containing a letter $a$ are those with boundary $a\inv u_a(A_1)$. Since $S$ is a surface, it follows that there are two 2-cells which abut this edge, both with boundary label $a\inv u_a(A_1)$, as shown in \cref{fig:internaledge}. Collapsing all such such 2-cells, thereby removing these $A$-edges, yields a surface diagram $S'$ with boundary $\gamma$, $\Area(S')\le\Area(S)$ and $(\Rad(S')\le k\cdot\Rad(S)$. Thus we see that $(\ol[f],\ol[g])$ is a homological AR pair for $\hpres[A_1]{S_0}$, where $\ol[f]\coloneqq f_2\simeq f_1\simeq f$ and $\ol[g]\coloneqq kg_2\cong g_1\cong g$.

    \begin{figure}[!ht]
    \centering
    \resizebox{0.4\textwidth}{!}{
    \begin{tikzpicture}[very thick, decoration={
    markings,
    mark=at position 0.5 with {\arrow{>}}}]
        \filldraw[fill=black!5!white] (4,2) 
            arc (0:360:2)
            node [pos=0.3] (t) {}
            node [pos=0.3, above left] {$v$} 
            node [pos=0.8] (u) {}
            node [pos=0.8, below right] {$va=vu_a(A_1)$};
        \fill (t.center) circle (0.08);
        \fill (u.center) circle (0.08);
        

        \draw[postaction=decorate] (4,2) arc (0:-10:2)
            node [right] {$u_a(A_1)$};
        \draw[postaction=decorate] (0,2) arc (180:190:2)
            node [left] {$u_a(A_1)$};
        \draw[postaction=decorate] (t.center) to 
            node [pos=0.5, right] {$a$}
            (u.center);
    \end{tikzpicture}
    }
    \caption{An internal $A$-edge.}
    \label{fig:internaledge}
\end{figure}
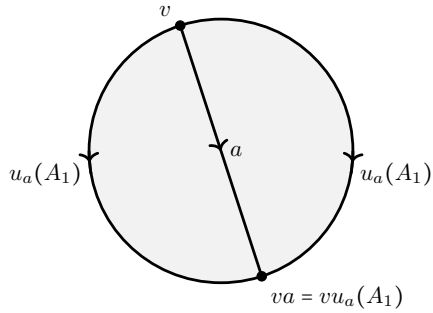
\end{proof}

\subsection{A Homological Finite Presentation for \texorpdfstring{$H$}{H}}
Now we are able to state an analogue of Gersten and Short's result for groups of type $FP_2$:
\begin{bigthm}\label{thm:homger}
    Let $H$ be an extension of a group $K$ of type $FP_2$ by a finitely generated free group $F_n$, so that we have the short exact sequence
    \begin{center}
        $1\ra K\ra H\ra F_n\ra 1$.
    \end{center}
    Then $H$ is of type $FP_2$, and if $(f,g)$ is a homological area-radius pair for $H$ with $f(n)\geq n$ for all $n$, there is a constant $C >1$ such that $\FA_K\preccurlyeq C^gf$.
\end{bigthm}

It is known that $H$ is necessarily of type $FP_2$, following Proposition 2.2 of \cite{kochloukova2016homologicalfinitenesspropertiesfibre}. Gersten and Short's proof opens by constructing a finite presentation for $H$ from one for $K$. Below, we claim that the analogous construction yields a homological finite presentation for $H$, thus giving an alternative proof that a free extension of a group of type $FP_2$ is also of type $FP_2$. Recall that every extension with free quotient is split, so $F_n$ is naturally a subgroup of $H$.

Conjugation of $K\le H$ by a free generator $t_i$ of $F_n$ gives an automorphism $\phi_i:K\ra K$. Let $\Phi_i$ be a lift of $\phi_i$ to a semigroup endomorphism as in \cref{not:bigphi}.
\begin{lem}\label{lem:hfp2}
    Let $H$ be an extension of a group $K=\pres[A]{R}$ of type $FP_2$ by a finitely generated free group $F_n$, so that we have the short exact sequence
        $$1\ra K\ra H\ra F_n\ra 1.$$
    If $\hpres[A]{R_0}$ is a homological finite presentation for $K$, then $$\hpres[A\cup \left\{ t_1,\ldots,t_n\right\}]{R_0\cup\left\{t_i\inv a_j t_i\Phi_i(a_j)\inv\,\middle|\,a_j\in A,i\in\left\{1,\ldots,n\right\}\right\}}$$ is a homological finite presentation for $H$.
\end{lem}
\begin{proof}
    Note first that the proof of Theorem B in \cite{gersten2002} tells us that $\pres[A\cup \left\{ t\right\}]{R\cup\left\{t\inv a_j t\Phi(a_j)\inv\,|\,a_j\in A\right\}}$ is a presentation for $H$.
    
    We can picture the coarse structure of the Cayley graph $\Gamma(H,A\cup\left\{t_i,\ldots,t_n\right\})$ as looking like the Cayley graph of the free group $F(t_1,\ldots,t_n)$ (the $2n$-valent tree), where each vertex is replaced by a copy of the Cayley graph $\Gamma(K,A)$ for $K$, with each copy connected to the $2n$ adjacent copies by edges labelled $t_i$. Each `vertex' of this 2n-valent tree corresponds to a coset of the subgroup $K$ in $H$.

        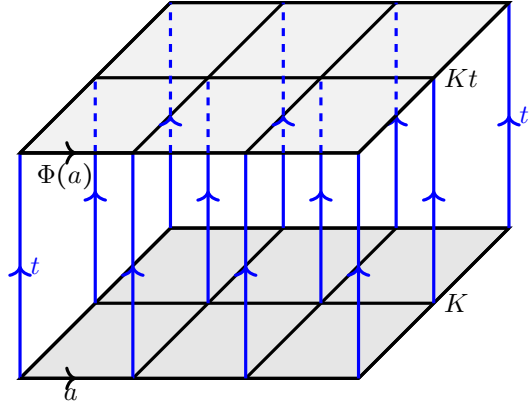
\begin{figure}[!ht]
    \centering
    \begin{tikzpicture}[very thick,decoration={
    markings,
    mark=at position 0.5 with {\arrow{>}}}]
        \filldraw[fill=black!10!white]
            (0,0) -- (4.5,0) -- (6.5,2) -- (2,2) -- cycle;
        \draw (0,0) to (4.5,0);
        \draw (1,1) to node [right, pos=1] {$K$} (5.5,1);
        \draw (2,2) to (6.5,2);
        \draw (0,0) to (2,2);
        \draw (1.5,0) to (3.5,2);
        \draw (3,0) to (5,2);
        \draw (4.5,0) to (6.5,2);

        \draw[blue,postaction=decorate] ($(0,0)+(0,0)$) to node [right,pos=0.5] {$t$} ($(0,0)+(0,3)$);
        \draw[blue,postaction=decorate] ($(1.5,0)+(0,0)$) to ($(1.5,0)+(0,3)$);
        \draw[blue,postaction=decorate] ($(3,0)+(0,0)$) to ($(3,0)+(0,3)$);
        \draw[blue,postaction=decorate] ($(4.5,0)+(0,0)$) to ($(4.5,0)+(0,3)$);
        \draw[blue,postaction=decorate] ($(1,1)+(0,0)$) to ($(1,1)+(0,3)$);
        \draw[blue,postaction=decorate] ($(2.5,1)+(0,0)$) to ($(2.5,1)+(0,3)$);
        \draw[blue,postaction=decorate] ($(4,1)+(0,0)$) to ($(4,1)+(0,3)$);
        \draw[blue,postaction=decorate] ($(5.5,1)+(0,0)$) to ($(5.5,1)+(0,3)$);
        \draw[blue,postaction=decorate] ($(2,2)+(0,0)$) to ($(2,2)+(0,3)$);
        \draw[blue,postaction=decorate] ($(3.5,2)+(0,0)$) to ($(3.5,2)+(0,3)$);
        \draw[blue,postaction=decorate] ($(5,2)+(0,0)$) to ($(5,2)+(0,3)$);
        \draw[blue,postaction=decorate] ($(6.5,2)+(0,0)$) to node [right,pos=0.5] {$t$} ($(6.5,2)+(0,3)$);

        \filldraw[fill=black!5!white]
            (0,3) -- (4.5,3) -- (6.5,5) -- (2,5) -- cycle;
        \draw ($(0,0)+(0,3)$) to ($(4.5,0)+(0,3)$);
        \draw ($(1,1)+(0,3)$) to node [right, pos=1] {$Kt$} ($(5.5,1)+(0,3)$);
        \draw ($(2,2)+(0,3)$) to ($(6.5,2)+(0,3)$);
        \draw ($(0,0)+(0,3)$) to ($(2,2)+(0,3)$);
        \draw ($(1.5,0)+(0,3)$) to ($(3.5,2)+(0,3)$);
        \draw ($(3,0)+(0,3)$) to ($(5,2)+(0,3)$);
        \draw ($(4.5,0)+(0,3)$) to ($(6.5,2)+(0,3)$);

        \draw[dashed,blue] ($(1,1)+(0,2)$) to ($(1,1)+(0,3)$);
        \draw[dashed,blue] ($(2.5,1)+(0,2)$) to ($(2.5,1)+(0,3)$);
        \draw[dashed,blue] ($(4,1)+(0,2)$) to ($(4,1)+(0,3)$);
        \draw[dashed,blue,->] ($(2,2)+(0,1)$) to ($(2,2)+(0,1.5)$);
        \draw[dashed,blue,->] ($(3.5,2)+(0,1)$) to ($(3.5,2)+(0,1.5)$);
        \draw[dashed,blue] ($(2,2)+(0,1.5)$) to ($(2,2)+(0,3)$);
        \draw[dashed,blue] ($(3.5,2)+(0,1.5)$) to ($(3.5,2)+(0,3)$);
        \draw[dashed,blue] ($(5,2)+(0,1.5)$) to ($(5,2)+(0,3)$);

        \draw[postaction=decorate] (0,0) to node [below, pos=0.45] {$a$} (1.5,0);
        \draw[postaction=decorate] (0,3) to node [below, pos=0.4] {$\Phi(a)$} (1.5,3);
    \end{tikzpicture}
    \caption{Sketch of part of $\Gamma(H,A\cup\left\{t\right\})$, where $n=1$.}
    \label{fig:layers}
\end{figure}
    
    We construct a 2-complex $X$ from $\Gamma(H,A\cup\left\{t_i,\ldots,t_n\right\})$ by gluing, for each vertex $x\in H$ and for each of the finitely many relations $r\in R_0\cup\left\{t_i\inv a_j t_i\Phi_i(a_j)\inv\,|\,a_j\in A, i\in\left\{1,\ldots,n\right\}\right\}$, a $2$-cell with boundary the loop labelled by $r$ beginning at $x$. Under this construction, the copies of $\Gamma(K,A)$ corresponding to cosets $Kw$ in $H$ have become copies of the homological Cayley complex $X(K,A,R_0)$. Importantly, each copy of the homological Cayley complex for $K$ has trivial first homology. It suffices to prove that $H_1(X)$ is trivial, for which we show that every $1$-cycle admits a filling by a $2$-chain. It suffices to check the result for a loop $\gamma$ in $X^{(1)}$.

    A loop $\gamma$ in $X^{(1)}$ can only intersect finitely many of the cosets $Kw$ where $w$ is a reduced word in $F(t_1,\ldots,t_n)$.
    A continuous loop $\gamma$ can only traverse from a coset $Kw$ to an adjacent one $Kwt_i$ or $Kwt_i\inv$ by an edge labelled $t_i$.
    Consider the set $S$ of the words $w$ such that $\gamma$ intersects $Kw$ non-trivially. $S$ must contain a unique word of minimal length. For suppose that $u\neq v$ are two words of minimal length, then there must exist a path along $\gamma$ from $Ku$ to $Kv$. If this path traverses an edge labelled $t_i$, then it moves into the coset $Kut_i$. We must have $|ut_i|=|u|+1$ by minimality of $|u|$. Any subsequent $t$-edges traversed can only move to $Ku'$ for $|u'|>|u|$ or back into $Ku$, where the above argument applies again. There is no way of changing any of the first $|u|$ letters of $u$ by traversing $t$-edges without moving into a coset $Ku'$ with $|u'|<|u|$, contradicting minimality. Set $\ol[w]$ to be the unique word of minimal length in $S$.
    
    Pick a word $w\in S$ of maximal length. If the last letter of $w$ is $t_i$, that is, $w=w'\ast t_i=w't_i$ for some strictly shorter word $w'$ of length $|w|-1$. Let $\Psi_i$ be a lift of $\phi_i\inv$ as in \cref{not:bigphi}. For each edge labelled $a_j$ that $\gamma$ traverses in $Kw$, there are $2$-chains in $Kw$ with boundary $a_j(\Phi_i\circ\Psi_i)(a_j)\inv$, and for each edge $a_k$ in $\Psi_i(a_j)$, there is a $2$-cell  $t\inv a_k t\Phi_i(a_k)\inv$. Taking all such $2$-chains and $2$-cells allows us to `push down' the part of the loop $\gamma$ to a homologous loop $\gamma'$ which is the same as $\gamma$ except the section in $Kw$ is replaced by a homologous section in $Kw'$. If instead the last letter of $w$ is $t_i\inv$, that is , $w=w'\ast t_i\inv=w't_i\inv$, then take, for all edges $a_j$ traversed by $\gamma$ in $Kw$, $2$-cells of the form $t\inv a_j t\Phi_i(a_j)\inv$. The collection of such $2$-cells allows us to `push up' the part of the loop $\gamma$ in $Kw$ to a homologous part in $Kw'$ to obtain a homologous loop $\gamma'$ which does not intersect $Kw$.

    Applying the process above repeatedly to all words in $S$ from longest to shortest eventually yields a loop $\ol[\gamma]$ contained in $K\ol[w]$. $\ol[\gamma]$ is the boundary of a $2$-chain in $K\ol[w]$. By taking the union of this $2$-chain with all the $2$-cells required to push $\gamma$ into $K\ol[w]$, we obtain a $2$-chain whose boundary is $\gamma$. Thus $X$ is a homological Cayley complex for $H$.
\end{proof}

\subsection{New Filling from an Automorphism}\label{subsec:newfill}
Section 3 of \cite{gersten2002} gives upper bounds for the area of a disc diagram after the application of an automorphism. The methods of Gersten and Short for their Lemmas 3.1 and 3.2 apply with minimal adjustments to the case of our surface diagram analogue.

Let $\mathcal{P}_0=\hpres[A]{R_0}$ be a homological finite presentation for a group $H$ of type $FP_2$, and let $\phi:H\ra H$ be a group automorphism. Let $\Phi$ and $\Psi$ be as in \cref{not:bigphi}. For a $1$-cycle $\gamma$ in $X(A,R_0)$, we denote by $\Phi(\gamma)$ the $1$-cycle in $X(A,R_0)$ obtained by mapping each vertex $x\in\gamma$ to $\Phi(x)$ and each edge labelled $a$ with initial vertex $x$ and terminal vertex $y$ to the path labelled by the word $\Phi(a)$ from $\Phi(x)$ to $\Phi(y)$. Observe that this is well-defined, since $\phi$ maps words representing the identity in $H$ to other such words, and thus so does $\Phi$. If $\gamma=\gamma_1+\cdots+\gamma_\ell$ is the decomposition of a $1$-cycle $\gamma$ into distinct loops $\gamma_i$, then $\Phi(\gamma)=\Phi(\gamma_1)+\cdots +\Phi(\gamma_\ell)$, and if a loop $\gamma_i$ has boundary word $w_i$, $\Phi(\gamma_i)$ has boundary word $\Phi(w_i)$.

\begin{lem}\label{lem:hareacap1}
    Let $H=\pres[A]{R}$ be a group of type $FP_2$ with homological finite presentation $\hpres[A]{R_0}$. Let $\gamma$ be a $1$-cycle in $X(A,R_0)^{(1)}$ and let $c=\sum_ia_i\sigma_i$ be a $2$-chain with $\del c=\gamma$. There exists a constant $C$ dependent only on $A, R_0$, and a $2$-chain $c'$ with $\del c'=\Phi(\gamma)$ and \emph{$\Area(c')\leq C\cdot \Area(c)$}.
\end{lem}
\begin{proof}
    Let $c^{(1)}$ denote the $1$-skeleton of the subcomplex of $X(A,R_0)$ consisting of the edges bounding each $2$-cell in $c$. Subdivide and relabel each edge in $c^{(1)}$, so that the edge previously labelled by a generator $a_i$ is now labelled by a word $\Phi(a_i)$. The outside boundary of this complex is the image $\Phi(\gamma)$, the loops which once bounded $2$-cells are labelled by words $\Phi(r)$ for $r\in R_0$. Each $\Phi(r)$ is a relation in $H$, so can be filled by a $2$-chain. Fix a $2$-chain $c_{\Phi(r)}$ for each $r\in R_0$, set 
        $$C=\max\left\{\Area(c_{\Phi(r)})\,\big|\,r\in R_0\right\},$$
    then by filling each of these loops with boundary labelled $\Phi(r)$ by the $2$-chain $c_{\Phi(r)}$, we obtain a $2$-chain filling $c'$ for $\Phi(w)$ with
        $$\Area(c')\leq C\cdot \Area(c).$$
\end{proof}

\begin{lem}\label{lem:hareacap2}
    Let $H=\pres[A]{R}$ be a group of type $FP_2$ with homological finite presentation $\hpres[A]{R_0}$. Let $\gamma$ be a $1$-cycle in $X(A,R_0)^{(1)}$ and let $c'=\sum_ia_i\sigma_i$ be a $2$-chain with $\del c'=\Phi(\gamma)$. There exist constants $C',C''$ dependent only on $A, R_0$, and a $2$-chain $c''$ with $\del c''=\gamma$ and \emph{$\Area(c'')\leq C'\cdot \Area(c')+C''\cdot |\gamma |$}.
\end{lem}
\begin{proof}
    Analogously to the proof of \cref{lem:hareacap1}, we obtain a filling $c_{\Psi\circ\Phi(\gamma)}$ for $\Psi\circ\Phi(\gamma)$ with area at most $C'\cdot\Area (c')$, where we fix fillings of the relations $\Psi\circ\Phi(r)$ and set
        $$C'=\max\left\{\Area(c_{\Psi\circ\Phi(r)})\,|\,r\in R_0\right\}.$$
    For each generator $a\in A$, fix a $2$-chain filling $c_a$ for the relation $a\inv\left(\Psi\circ\Phi\right)(a)$. Let
        $$C''=\max\left\{\Area(c_a)\,|\,a\in A\right\}.$$
    Attaching such $2$-chains to the boundary of $c_{\Psi\circ\Phi(\gamma)}$ - there will be one for each edge in the boundary of $\gamma$ - yields a $2$-chain filling $c''$ for $\gamma$ with 
        $$\Area(c'')\leq C'\cdot \Area(c')+C''\cdot |\gamma|.$$
\end{proof}
\begin{rem}
    As the above lemmas hold for any filling $c$, we must then have both
    \begin{align*}
        \HArea(\Phi(\gamma))&\leq C\cdot \HArea(\gamma),\\
        \HArea(\gamma)&\leq C'\cdot \HArea(\Phi(\gamma))+C''\cdot |\gamma|.
    \end{align*}
\end{rem}
\begin{rem}\label{cor:sdareacap}
    Lemmas \ref{lem:hareacap1}, \ref{lem:hareacap2} above can be applied to surface diagram fillings as follows. For a $1$-cycle $\gamma$ in $X(A,R_0)^{(1)}$, there exists a surface diagram filling $(S,\pi)$ for $\gamma$, then $\pi(S)$ is exactly some $2$-chain $c$ with $\del c=\gamma$, which can be exchanged using \cref{lem:hareacap1} for a $2$-chain $c'$ with $\del c'=\Phi(\gamma)$, which has a corresponding surface diagram filling $(S',\pi')$ with $\Area(c')=\Area(S')$ by \cref{lem:sdexists} and \cref{rem:samearea}. Thus the area of the new surface diagram $S'$ is bounded in terms of the area of $S$ as in \cref{lem:hareacap1}. The process for \cref{lem:hareacap2} is similar.
\end{rem}

\subsection{\texorpdfstring{$t$}{t}-rings on Surface Diagrams}\label{sec:tring}
As for Gersten and Short's disc diagrams, one observes that any instance of an edge labelled by some $t_i$ must be contained in the interior of a `$t$-ring' structure.

When $K$ is finitely presented, a loop always admits a filling by a disc diagram, and one can always `push down' a $t$-ring by exchanging a disc diagram filling of a word $w$ for a disc diagram filling of $\Phi(w)$, or vice versa. When the filling comes from a surface diagram $(S,\pi)$, these $t$-rings still exist, but there is a possibility that a $t$-ring may not be nullhomotopic in $S$. Consider the case pictured in \cref{fig:badtring}, where a $t$-ring wraps around a handle of $S$. A priori, in this case we cannot apply \cref{lem:hareacap1} or \cref{lem:hareacap2} to push down our loop, since there is no surface diagram filling for either $w$ or $\Phi(w)$.

\begin{figure}[!ht]
    \centering
    \begin{subfigure}{0.4\textwidth}
    \centering
    \resizebox{\textwidth}{!}{
    \begin{tikzpicture}[very thick,decoration={
    markings,
    mark=at position 0.5 with {\arrow{>}}},scale=2.25
    ] 
        \filldraw[fill=black!3!white, postaction=decorate]
            (1,2) ..
            controls ++(90:-0.8)
            and ++ (90:-0.8) ..
            (5,2);
        \filldraw[fill=black!5!white]
            (5,2) ..
            controls ++(90:2)
            and ++ (90:2) ..
            (1,2);
        \filldraw[fill=black!3!white,dashed, postaction=decorate]
            (5,1.98) ..
            controls ++(90:0.5)
            and ++ (90:0.5) ..
            (1,1.98);
        \fill[fill=black!5!white]
            (1.75,3.1) ..
            controls ++(90:2)
            and ++(90:2) ..
            (4.25,3.1)..
            controls ++(0:0)
            and ++(0:0) ..
            (3.75,3.3) ..
            controls ++(90:1.1)
            and ++(90:1.1) ..
            (2.25,3.3);
        \draw
            (2.4,3) ..
            controls ++(0:-0.1)
            and ++(90:-0.1) ..
            (2.25,3.3) ..
            controls ++(90:1.1)
            and ++(90:1.1) ..
            node [pos=0.1] (a) {}
            node [pos=0.3] (b) {}
            (3.75,3.3)..
            controls ++(90:-0.1)
            and ++(0:0.1) ..
            (3.6,3);
        \draw
            (1.6,2.8) ..
            controls ++(0:0.1)
            and ++(90:-0.1) ..
            (1.75,3.1) ..
            controls ++(90:2)
            and ++(90:2) ..
            node [pos=0.1] (a') {}
            node [pos=0.25] (b') {}
            (4.25,3.1)..
            controls ++(90:-0.1)
            and ++(0:-0.1) ..
            (4.4,2.8);
        \draw[postaction=decorate]
            (a'.center) ..
            controls ++(90:-0.1)
            and ++(90:-0.1) ..
            node [pos=0.6] (c) {} (a.center)
            node [pos=0.5, below] {$\Phi(w)$};
        \draw[dashed,postaction=decorate]
            (a.center) ..
            controls ++(90:0.1)
            and ++(90:0.1) ..
            node [pos=0.6] (d) {}
            (a'.center);
        \draw[postaction=decorate]
            (b'.center) ..
            controls ++(210:0.1)
            and ++(90:-0.1) ..
            node [pos=0.6] (c') {} (b.center);
        \draw[dashed,postaction=decorate]
            (b.center) ..
            controls ++(90:0.1)
            and ++(90:0.1) ..
            node [pos=0.6] (d') {}
            node [pos=0.5, above right] {$w$}
            (b'.center);
        \draw[blue,postaction=decorate] 
            (c.center) ..
            controls ++(90:0.2)
            and ++(60:-0.2) ..
            node [above, pos=0.5] {$t$}
            (c'.center);
        \draw[blue, dashed,postaction=decorate] 
            (d.center) ..
            controls ++(90:0.2)
            and ++(60:-0.2) ..
            (d'.center);
    \end{tikzpicture}
    }
    \caption{A worrying `$t$-ring' placement on $S$.}
    \label{fig:badtring}
    \end{subfigure}%
    \begin{subfigure}{0.4\textwidth}
    \centering
    \resizebox{\textwidth}{!}{
    \begin{tikzpicture}[very thick,decoration={
    markings,
    mark=at position 0.5 with {\arrow{>}}},scale=2.25
    ] 
        \filldraw[fill=black!3!white, postaction=decorate]
            (1,2) ..
            controls ++(90:-0.8)
            and ++ (90:-0.8) ..
            (5,2);
        \filldraw[fill=black!5!white]
            (5,2) ..
            controls ++(90:2)
            and ++ (90:2) ..
            (1,2);
        \filldraw[fill=black!3!white,dashed, postaction=decorate]
            (5,1.98) ..
            controls ++(90:0.5)
            and ++ (90:0.5) ..
            (1,1.98);
        \fill[fill=black!5!white]
            (1.75,3.1) ..
            controls ++(90:2)
            and ++(90:2) ..
            (4.25,3.1)..
            controls ++(0:0)
            and ++(0:0) ..
            (3.75,3.3) ..
            controls ++(90:1.1)
            and ++(90:1.1) ..
            (2.25,3.3);
        \draw
            (2.4,3) ..
            controls ++(0:-0.1)
            and ++(90:-0.1) ..
            (2.25,3.3) ..
            controls ++(90:1.1)
            and ++(90:1.1) ..
            node [pos=0.1] (a) {}
            node [pos=0.3] (b) {}
            node [pos=0.9] (x) {}
            node [pos=0.7] (y) {}
            (3.75,3.3)..
            controls ++(90:-0.1)
            and ++(0:0.1) ..
            (3.6,3);
        \draw
            (1.6,2.8) ..
            controls ++(0:0.1)
            and ++(90:-0.1) ..
            (1.75,3.1) ..
            controls ++(90:2)
            and ++(90:2) ..
            node [pos=0.1] (a') {}
            node [pos=0.25] (b') {}
            node [pos=0.9] (x') {}
            node [pos=0.75] (y') {}
            (4.25,3.1)..
            controls ++(90:-0.1)
            and ++(0:-0.1) ..
            (4.4,2.8);
        \draw[postaction=decorate]
            (a'.center) ..
            controls ++(90:-0.1)
            and ++(90:-0.1) ..
            node [pos=0.6] (c) {} (a.center)
            node [pos=0.5, below] {$\Phi(w)$};
        \draw[dashed,postaction=decorate]
            (a.center) ..
            controls ++(90:0.1)
            and ++(90:0.1) ..
            node [pos=0.6] (d) {}
            (a'.center);
        \draw[postaction=decorate]
            (b'.center) ..
            controls ++(210:0.1)
            and ++(90:-0.1) ..
            node [pos=0.6] (c') {} (b.center);
        \draw[dashed,postaction=decorate]
            (b.center) ..
            controls ++(90:0.1)
            and ++(90:0.1) ..
            node [pos=0.6] (d') {}
            node [pos=0.5, above right] {$w$}
            (b'.center);
        \draw[blue,postaction=decorate] 
            (c.center) ..
            controls ++(90:0.2)
            and ++(60:-0.2) ..
            node [above, pos=0.5] {$t$}
            (c'.center);
        \draw[blue, dashed,postaction=decorate] 
            (d.center) ..
            controls ++(90:0.2)
            and ++(60:-0.2) ..
            (d'.center);

        \draw[postaction=decorate]
            (x.center) ..
            controls ++(90:-0.1)
            and ++(90:-0.1) ..
            node [pos=0.3] (z) {} (x'.center)
            node [pos=0.6, below] {$\Phi(w')$};
        \draw[dashed,postaction=decorate]
            (x'.center) ..
            controls ++(90:0.1)
            and ++(90:0.1) ..
            node [pos=0.3] (w) {}
            (x.center);
        \draw[postaction=decorate]
            (y.center) ..
            controls ++(60:-0.1)
            and ++(90:-0.1) ..
            node [pos=0.4] (z') {} (y'.center);
        \draw[dashed,postaction=decorate]
            (y'.center) ..
            controls ++(90:0.1)
            and ++(90:0.1) ..
            node [pos=0.3] (w') {}
            node [pos=0.5, above left] {$w'$}
            (y.center);
        \draw[blue,postaction=decorate] 
            (z.center) ..
            controls ++(90:0.2)
            and ++(120:-0.2) ..
            node [above, pos=0.5] {$t$}
            (z'.center);
        \draw[blue, dashed,postaction=decorate] 
            (w.center) ..
            controls ++(90:0.2)
            and ++(120:-0.2) ..
            (w'.center);
    \end{tikzpicture}
    }
    \caption{The corresponding `$t$-cycle' on $S$.}
    \label{fig:goodtring}
    \end{subfigure}
    \caption{}
\end{figure}

We overcome this by considering `$t$-cycles' in place of $t$-rings. Observe that the Cayley graph for $H$ with respect to the generating set of \cref{lem:hfp2} consists of a copy of $K$ for each element of $F_n$, corresponding to the different cosets $Kw$ in $H$. These coset subgraphs are joined to one another corresponding to the Cayley graph of $F_n$, that is, the only edges between $Kw$ and $Kwt_i$ are labelled $t_i$. As such, in the case of \cref{fig:badtring}, the only paths from $Kwt$ to $Kw$ must pass back along some $t$-edge, so there must be another $t$-ring on the other side of the handle, as pictured in \cref{fig:goodtring}.

Generalising this picture, we define a homological analogue of the $t$-ring.

\begin{defn}
    Let $(S,\pi)$ be a filling for the $1$-cycle $\gamma$ in $X(A,R_0)$. Consider $\pi\inv(Kw)$, the part of $S$ that maps into the coset $Kw$, and observe that its boundary is a collection of circles, each corresponding to some $t_i$-ring. For some fixed stable letter $t_i$, and each reduced word $w\in F(A)$ with $wt_i$ also reduced, we call the collection of $2$-cells whose boundary has label some $t_ia_jt_i\Phi(a_j)\inv$ and which intersects the subsurfaces $\pi\inv(Kw)$, $\pi\inv(Kwt_i)$, the \emph{$t_i$-cycle} joining $Kw$ to $Kwt_i$, or simply a $t$-cycle when the particular choice of stable letter $t_i$ and reduced word $w$ is irrelevant. 
    
    When $wt_i^{\pm 1}$ is a reduced word, we call the collection of loops in $S$ corresponding to the boundary in $Kw$ the \textit{inner boundary} $\del_{\In}$ and that for $Kwt_i^{\pm 1}$ the \textit{outer boundary} $\del_{\out}$.
\end{defn}
\begin{rem}
    Following \cref{cor:sdareacap}, observe that Lemmas \ref{lem:hareacap1}, \ref{lem:hareacap2} allow us to exchange a $K$-filling for the outer boundary of a $t$-cycle, for a $K$-filling for the inner boundary, with a linear bound on the area growth.
\end{rem}

\begin{lem}[$t$-cycles are unique]\label{gettcycle}
    Let $c$ be a finite $2$-chain in 
    \begin{center}
        $X=X(A\cup \left\{ t_1,\ldots,t_n\right\}, R_0\cup\left\{t_i\inv a_j t_i\Phi_i(a_j)\inv\right\})$
    \end{center} 
    whose boundary is a $1$-cycle $\gamma$ with no $t$-edges. Then every edge labelled $t_i$ in $c$ is contained in a unique $t_i$-cycle whose boundary consists of two $1$-cycles, each contained in one of the subcomplexes $Kw$, $Kwt_i$ for some $w\in F(t_1,\ldots,t_n)$.
\end{lem}
\begin{proof}
    Consider a surface diagram $(S,\pi)$ for $c$. Consider an arbitrary $t$-edge with label $t_i$ in $S$. The initial vertex of this edge lies in $\pi\inv(Kw)$ for some reduced word $w\in F(t_1,\ldots,t_n)$ and the terminal vertex lies in $\pi\inv(Kwt_i)$. So this $t_i$-edge is contained in the $t_i$-cycle whose inner and outer boundaries lie in $\pi\inv(Kw)$ and $\pi\inv(Kwt_i)$. In the Cayley graph, an edge labelled $t_i$ joins two unique cosets $Kw,Kwt_i$, so this $t_i$-edge cannot be contained in any other $t$-cycle, since by definition there is a unique $t$-cycle joining any two adjacent cosets of $K$ in $H$.
\end{proof}

\subsection{The Homological Isoperimetric Inequality}
Recall the hypotheses of \cref{thm:homger}.

\textbf{Theorem \ref{thm:homger}.}\emph{
    Let $H$ be an extension of a group $K$ of type $FP_2$ by a finitely generated free group $F_n$, so that we have the short exact sequence
    \begin{center}
        $1\ra K\ra H\ra F_n\ra 1$.
    \end{center}
    Then if $(f,g)$ is a homological area-radius pair for $H$ with $f(n)\geq n$ for all $n$, there is a constant $C >1$ such that $\FA_K\preccurlyeq C^gf$.
}
\begin{proof}
    Let $K\cong\langle A\,|\,R\rangle$ have a homological finite presentation $\hpres[A]{R_0}$. Let $X$ denote the homological Cayley complex for $H$ with respect to the presentation given in \cref{lem:hfp2}. Let $\gamma$ be a $1$-cycle in $\Gamma(K,A)$. We embed $\gamma$ into $X$ by identifying $\Gamma(K,A)$ with the subcomplex corresponding to $K$ in the $1$-skeleton $X^{(1)}=\Gamma(H,A\cup\left\{t_1,\ldots,t_n\right\})$.

    Let $c=\sum_ia_i\sigma_i$ be a $2$-chain with $\del c=\gamma$ in $X$ with area no more than $f(|\gamma|)$. Let $(S,\pi)$ be a surface diagram filling for $c$. In $S$, the $2$-cells corresponding to relations $t\inv a_j t\Phi(a_j)\inv$ form $t$-cycles. Each $t$-cycle is associated to one of the letters $t_i$ and connects regions $\pi\inv(Kw)$ and $\pi\inv(Kwt_i)$ of $S$. Let $W_c$ be the set of all reduced words $w\in F(t_1,\ldots,t_n)$ such that the filling $c$ has non-empty intersection with the subcomplex of $X$ corresponding to the coset $Kw$.

    Let $w$ be a word of maximal length $k$ in $W_c$, say the last letter of $w$ is $t_i^{\pm 1}$, so $w=w' t_i^{\pm 1}$, where $w'\in F(t_1,\ldots,t_n)$ has length $k-1$. We can restrict $(S,\pi)$ to a surface diagram filling $\left(S_{\out}\coloneqq\pi\inv(Kw),\pi|_{S_{\out}}\right)$ for the outer boundary $\del_{\out}$ of the $t_i$-cycle connecting $Kw'$ to $Kw$. $S_{\out}$ has no interior $t$-edges, so is a $K$-filling, since if it did, this would join $Kw$ to some $Kwt_j^{\pm 1}$, where $wt_j^{\pm 1}$ is a reduced word strictly longer than $w$. Here we separate cases depending whether the final letter of $w$ is a stable letter or its inverse. 
    
    If $w$ ends in $t_i$, the labels on the outer boundary $\del_{\out}$ can be obtained by applying $\Phi_i$ to the inner boundary $\del_{\In}$ of the $t_i$-cycle, we write $\del_{\out}=\Phi_i(\del_{\In})$. Applying \cref{lem:hareacap2} as described in \cref{cor:sdareacap}, we can exchange the $K$-filling $S_{\out}$ of $\del_{\out}$ and the $t_i$-cycle itself for a $K$-filling $S_{in}$ of $\del_{\In}$ whose area $\Area (S_{\In})\leq C'\cdot\Area (S_{\out})+C''\cdot|\del_{\out}|$, where $C'$ and $C''$ are constants as in \cref{lem:hareacap2}.

    If $w$ ends in $t_i\inv$, we have that $\del_{\In}=\Phi_i(\del_{\out})$, and applying \cref{lem:hareacap1} as described in \cref{cor:sdareacap} tells us that we can exchange the $K$-filling $S_{\out}$ for $\del_{\out}$ and the $t_i$-cycle itself with a $K$-filling $S_{\In}$ for $\del_{\In}$ whose area $\Area (S_{\In})\leq C\cdot\Area (S_{\out})$, where $C$ is a constant as in \cref{lem:hareacap1}.
    
    Repeating this process for all $w\in W_C$ of length $k$, we have removed all words of length $k$ in $W_c$ from the filling by eliminating their corresponding $t_i$-cycles, that is, we have obtained a new surface diagram filling for $\gamma$ that intersects only the cosets corresponding to words in $W_c$ of length strictly less than $k$.
    
    Since $S$ is a surface, any edge in $S$ appears in the boundary of at most 2 $2$-cells in $S$, so can occur in no more than $2$ $t$-cycles. It follows that twice the number of edges in $S$ is an upper bound on the length $|\del_{\out}|$. Let $\rho$ be the length of the longest relation in $R_0\cup\left\{t_i\inv a_j t_i\Phi_i(a_j)\inv\right\}$, then $|\del_{\out}|\leq |\gamma|+2\rho\Area (S)$. Applying this process simultaneously to all words of maximal length at worst multiplies the area of $S$ by a constant $M$ and adds at most $Mf(|\gamma|)$, where it suffices to take $M= \max\left\{C,C',C''(2\rho+1),1\right\}$. The maximal length $k$ of a word $w\in W_c$ is bounded above by the radius $g(|\gamma|)$, hence one can apply the above procedure $k\leq g(|\gamma|)$ times to obtain a surface diagram filling $(S',\pi')$ for $\gamma$ whose image is contained in $K$ and whose area is no more than
    $$\mu^{g(|\gamma|)}\left(f\left(|\gamma|\right)\right)
        \leq M^{g(|\gamma|)+1}f(|\gamma|)\leq (M^2)^{g(|\gamma|)}f(|\gamma|)$$
    where $\mu(x)=Mx+Mf(|\gamma|)$ is the upper bound on the area increase given by each iteration of Lemmas \ref{lem:hareacap1}, \ref{lem:hareacap2}. Thus it suffices to take the constant $C\geq M^2$ in the statement of the theorem.
\end{proof}

    \subsection{Hyperbolic Groups}
A key motivation for Gersten and Short's work was to better understand isoperimetric inequalities for subgroups of hyperbolic groups. For our purposes, hyperbolic groups are characterised by satisfying a linear isoperimetric inequality, that is, a finitely presented group $H$ is hyperbolic if and only if $\delta_H(n)\simeq n$.

Gersten and Short's general result is applied in the case $H$ hyperbolic in Theorem A of \cite{gersten2002}:
\textbf{\cref{thm:gerhyp}}\emph{
    Given an extension
    \begin{center}
        $1\ra K\ra H\ra F_n\ra 1$
    \end{center}
    where $F_n$ is free and $K$ is finitely presented, then if $H$ is hyperbolic, $K$ satisfies a polynomial isoperimetric inequality.
}

For hyperbolic groups, Gersten and Short give, in Lemma 2.2 of \cite{gersten2002}, an area-radius pair.
\begin{lem}\label{lem:hypbd}\emph{(\cite{gersten2002})}
    Let $\mathcal{P}=\pres[A]{R}$ be a finite presentation of a hyperbolic group $H$, then there are constants $B,C>0$ such that for any nullhomotopic word $w\in\ol[A]^*$ with $\ell(w)\geq 1$, there is a Van Kampen diagram over $\mathcal{P}$ with area at most $B\ell(w)(\log_2(\ell(w))+1)$ and radius at most $C(\log_2(\ell(w))+1)$.
\end{lem}
Applying this AR pair to \cref{thm:gergen} gives us \cref{thm:gerhyp}. After noting that an AR pair in the sense of \cite{gersten2002} is also a homological AR pair, we can obtain the following consequence of \cref{thm:homger}:
\begin{bigthm}\label{thmB}
    Let $H$ be an extension of a group $K$ of type $FP_2$ by a finitely generated free group $F_n$, so that we have the short exact sequence
    \begin{center}
        $1\ra K\ra H\ra F_n\ra 1$.
    \end{center}
    Then if $H$ is hyperbolic, $FA_K$ is bounded above by a polynomial.
\end{bigthm}

\bibliographystyle{alpha}
\bibliography{refs}

@article{gersten2002,
author={Gersten, Steve
and Short, Hamish},
title={Some Isoperimetric Inequalities for Kernels of Free Extensions},
journal={Geometriae Dedicata},
year={2002},
month={Jul},
day={01},
volume={92},
number={1},
pages={63-72},
issn={1572-9168},
doi={10.1023/A:1019659912872},
url={https://doi.org/10.1023/A:1019659912872}
}

@misc{brady2021,
      title={Homological {Dehn} functions of groups of type {$FP_2$}}, 
      author={Noel Brady and Robert Kropholler and Ignat Soroko},
      year={2021},
      eprint={2012.00730},
      archivePrefix={arXiv},
      primaryClass={math.GR},
      url={https://arxiv.org/abs/2012.00730}, 
}

@Inbook{gromov1987,
author="Gromov, M.",
title="Hyperbolic Groups",
bookTitle="Essays in Group Theory",
year="1987",
publisher="Springer New York",
address="New York, NY",
pages="75--263",
isbn="978-1-4613-9586-7",
doi="10.1007/978-1-4613-9586-7_3",
url="https://doi.org/10.1007/978-1-4613-9586-7_3"
}

@article{gersten1996,
author = {Gersten, S.M.},
journal = {Geometric and functional analysis},
number = {2},
pages = {301-345},
title = {Preservation and Distortion of Area in Finitely Presented Groups.},
url = {http://eudml.org/doc/58228},
volume = {6},
year={1996},
}

@article{gerstendim2,
author = {Gersten, S. M.},
title = {Subgroups of Word Hyperbolic Groups in Dimension 2},
journal = {Journal of the London Mathematical Society},
volume = {54},
number = {2},
pages = {261-283},
doi = {https://doi.org/10.1112/jlms/54.2.261},
url = {https://londmathsoc.onlinelibrary.wiley.com/doi/abs/10.1112/jlms/54.2.261},
eprint = {https://londmathsoc.onlinelibrary.wiley.com/doi/pdf/10.1112/jlms/54.2.261},
year = {1996}
}

@article{brady1999,
author = {Brady, Noel},
title = {Branched Coverings of Cubical Complexes and Subgroups of Hyperbolic Groups},
journal = {Journal of the London Mathematical Society},
volume = {60},
number = {2},
pages = {461-480},
doi = {https://doi.org/10.1112/S0024610799007644},
url = {https://londmathsoc.onlinelibrary.wiley.com/doi/abs/10.1112/S0024610799007644},
eprint = {https://londmathsoc.onlinelibrary.wiley.com/doi/pdf/10.1112/S0024610799007644},
year = {1999}
}

@article{rips1982,
    author = {Rips, E.},
    title = {Subgroups of small Cancellation Groups},
    journal = {Bulletin of the London Mathematical Society},
    volume = {14},
    number = {1},
    pages = {45-47},
    year={1982},
    month = {01},
    issn = {0024-6093},
    doi = {10.1112/blms/14.1.45},
    url = {https://doi.org/10.1112/blms/14.1.45},
    eprint = {https://academic.oup.com/blms/article-pdf/14/1/45/881637/14-1-45.pdf},
}

@article{bowditch1995,
author = {B. H. Bowditch},
title = {{A short proof that a subquadratic isoperimetric inequality implies a linear one.}},
volume = {42},
journal = {Michigan Mathematical Journal},
number = {1},
publisher = {University of Michigan, Department of Mathematics},
pages = {103 -- 107},
year = {1995},
doi = {10.1307/mmj/1029005156},
URL = {https://doi.org/10.1307/mmj/1029005156}
}

@misc{kropholler2021hyperbolicgroupsfinitelypresented,
      title={Hyperbolic groups with almost finitely presented subgroups}, 
      author={Robert Kropholler and Federico Vigolo},
      year={2021},
      eprint={1809.10594},
      archivePrefix={arXiv},
      primaryClass={math.GR},
      url={https://arxiv.org/abs/1809.10594}, 
}

@misc{arenas2022linearisoperimetricfunctionssurfaces,
      title={Linear isoperimetric functions for surfaces in hyperbolic groups}, 
      author={Macarena Arenas and Daniel T. Wise},
      year={2022},
      eprint={2205.10096},
      archivePrefix={arXiv},
      primaryClass={math.GT},
      url={https://arxiv.org/abs/2205.10096}, 
}

@Article{bestvina1997,
author={Bestvina, Mladen
and Brady, Noel},
title={Morse theory and finiteness properties of groups},
journal={Inventiones mathematicae},
year={1997},
month={Aug},
day={01},
volume={129},
number={3},
pages={445-470},
issn={1432-1297},
doi={10.1007/s002220050168},
url={https://doi.org/10.1007/s002220050168}
}

@misc{kochloukova2016homologicalfinitenesspropertiesfibre,
      title={Homological finiteness properties of fibre products}, 
      author={Dessislava H. Kochloukova and Francismar Ferreira Lima},
      year={2016},
      eprint={1611.03759},
      archivePrefix={arXiv},
      primaryClass={math.GR},
      url={https://arxiv.org/abs/1611.03759}, 
}

\end{document}